\documentclass{llncs}
\pdfoutput=1
 
\usepackage{epsfig}
\usepackage{amsmath}
\usepackage{amsfonts}
\usepackage{amssymb}
\usepackage{enumerate}

\usepackage{color}
\usepackage{mathrsfs}
\usepackage{graphicx}
\usepackage{subfig}

\usepackage[lined, linesnumbered, ruled]{algorithm2e}

\newcommand{\CA}{\mbox{$\mathsf{CA}$}}
\newcommand{\CAN}{\mbox{$\mathsf{CAN}$}}
\newcommand{\PCA}{\mbox{$\mathsf{PCA}$}}
\newcommand{\APCA}{\mbox{$\mathsf{APCA}$}}

\newcommand{\squishlisttwo}{
 \begin{list}{$-$}
  { \setlength{\itemsep}{0pt}
    \setlength{\parsep}{0pt}
    \setlength{	\topsep}{0pt}
    \setlength{\partopsep}{0pt}
    \setlength{\leftmargin}{2em}
    \setlength{\labelwidth}{1.5em}
    \setlength{\labelsep}{0.5em} } }

\newcommand{\squishend}{
  \end{list}  }

\title{Partial Covering Arrays: Algorithms and Asymptotics}
\toctitle{Partial Covering Arrays: Algorithms and Asymptotics}
\titlerunning{Partial Covering Arrays: Algorithms and Asymptotics}

\author{Kaushik Sarkar\inst{1} \and Charles J. Colbourn\inst{1} \and Annalisa De Bonis\inst{2} \and Ugo Vaccaro\inst{2}}

\institute{CIDSE, Arizona State University, U.S.A. \and Dipartimento di Informatica, University of Salerno, Italy}

\begin{document}


\maketitle

\begin{abstract}
A covering array $\CA(N;t,k,v)$ is an $N\times k$ array  with entries in $\{1, 2, \ldots , v\}$,  for which \emph{every} $N\times t$ subarray  contains \emph{each} $t$-tuple of  $\{1, 2, \ldots , v\}^t$ among its rows.
Covering arrays find application in interaction testing, including software and hardware testing,  advanced materials development, and  biological systems. 
A central question  is to determine or bound $\CAN(t,k,v)$, the minimum number $N$ of rows  of a $\CA(N;t,k,v)$. 
The well known bound $\CAN(t,k,v)=O((t-1)v^t\log k)$ is not too far from being asymptotically optimal. 
Sensible relaxations of the covering requirement arise when (1) the set $\{1, 2, \ldots , v\}^t$ need only be contained  among the rows of \emph{at least}  $(1-\epsilon)\binom{k}{t}$ of the $N\times t$ subarrays  and (2) the rows of \emph{every} $N\times t$ subarray  need only contain a (large) \emph{subset} of $\{1, 2, \ldots , v\}^t$.
In this paper, using probabilistic methods, significant improvements on the covering array upper bound are established for both relaxations, and for the conjunction of the two.  
In each case,  a randomized algorithm  constructs such arrays in expected polynomial time.
\end{abstract}

\section{Introduction} 
Let $[n]$ denote the set $\{1,2,\ldots,n\}$. 
Let $N,t,k,$ and $v$ be integers such that $k \ge t \ge 2$ and $v \ge 2$. 
Let $A$ be an $N \times k$ array where each entry is from the set $[v]$. 
For $I = \{j_1, \ldots, j_\rho\} \subseteq [k]$ where $j_1<\ldots<j_\rho$, let $A_I$ denote the $N \times \rho$ array in which $A_I(i,\ell) = A(i,j_\ell)$ for $1 \le i \le N$ and $1 \le \ell \le \rho$; $A_I$ is the projection of $A$ onto the columns in $I$.

A \emph{covering array} $\CA(N;t,k,v)$ is an $N \times k$ array $A$ with each entry  from $[v]$ so that  for each $t$-set of columns $C \in {[k] \choose t}$, each $t$-tuple $x \in [v]^t$ appears as a row in $A_C$.
The smallest $N$ for which a $\CA(N;t,k,v)$ exists is denoted by $\CAN(t,k,v)$.

Covering arrays find important application in software and hardware testing (see \cite{KKL2013} and references therein).
Applications of covering arrays also arise in experimental testing for advanced materials \cite{Caw2003},
inference of interactions that regulate gene expression \cite{Sha2001},
fault-tolerance of parallel architectures \cite{GHLS1993},  
synchronization of robot behavior \cite{Ha2005}, 
drug screening \cite{tong},
and  learning  of boolean functions \cite{dam}.
Covering arrays have been studied using different nomenclature,  as qualitatively independent partitions \cite{GKV93},  $t$-surjective arrays \cite{CKMZ1983}, and   $(k,t)$-universal sets \cite{Jukna}, among  others.
Covering arrays are closely related to  hash families \cite{Co2011} and orthogonal arrays \cite{Co2004}.

\section{Background and Motivation}

The exact or approximate determination of $\CAN(t,k,v)$ is central in applications of covering arrays, but remains an open problem.
For fixed $t$ and $v$, only when $t=v=2$ is $\CAN(t,k,v)$ known precisely for infinitely many values of $k$.
Kleitman and Spencer \cite{KlSp1973} and Katona \cite{Ka1973} independently proved that the 
largest $k$ for which a $\CA(N;2,k,2)$ exists satisfies $k=\binom{N-1}{\lceil N/2\rceil}.$
When $t=2$, 
Gargano, K\H{o}rner, and Vaccaro \cite{GKV93} establish that
\begin{equation}\label{gkv2}
\CAN(2,k,v) =\frac{v}{2}\log k(1+\mbox{o}(1)).
\end{equation}
(We write $\log$ for logarithms base 2, and $\ln$ for natural logarithms.)
Several researchers \cite{BB88,CKMZ1983,GSS96,GY2006} establish a general asymptotic
upper bound on $\CAN(t,k,v)$:
\begin{equation}\label{general1}
\CAN(t,k,v) \leq \frac{t-1}{\log\frac{v^t}{v^t-1}}\log k(1+\mbox{o}(1)).
\end{equation}
A slight improvement on (\ref{general1}) has recently been proved  \cite{FS2015,sarkar16}.
An (essentially) equivalent but more convenient form  of (\ref{general1}) is:
\begin{equation}\label{eq:can-upper}
\CAN(t,k,v) \le (t-1)v^t \log k(1+o(1)).
\end{equation}
A lower bound on $\CAN(t,k,v)$ results from the inequality $\CAN(t,k,v) \ge v  \cdot \CAN(t-1,k-1,v)$ obtained by derivation, together with (\ref{gkv2}), to establish that  $\CAN(t,k,v) \ge v^{t-2} \cdot \CAN(2,k-t+2,v) = v^{t-2}\cdot \frac{v}{2}\log(k-t+2)(1+\mbox{o}(1))$. 
When $\frac{t}{k} < 1$,  we obtain:
\begin{equation}\label{eq:can-lower}
\CAN(t,k,v) = \Omega(v^{t-1}\log k).
\end{equation} 

Because (\ref{eq:can-lower}) ensures that the number of rows in covering arrays can be considerable,
researchers have suggested the need for relaxations in which not all interactions must be covered \cite{Chen,HR2004,K2013+,MKTK201} in order to reduce the number of rows.
The practical relevance is that each row corresponds to a test  to be performed,
adding to the cost of testing. 

For example, an array \emph{covers a $t$-set of columns} when it covers each of the $v^t$ interactions on this $t$-set.
Hartman and Raskin \cite{HR2004} consider arrays with a fixed number of rows that
cover the \emph{maximum} number of  $t$-sets of columns.
A similar question was also considered in \cite{MKTK201}.
In \cite{K2013+,MKTK201}  a more refined measure of the (partial) coverage  of an $N\times k$ array $A$ is introduced.
For a given  $q\in [0,1]$, let $\alpha(A,q)$ be the number of $N\times t$ submatrices of $A$ with the property
that at least $qv^t$ elements of $[v]^t$
appear in their set of rows; the \emph{$(q,t)$-completeness} of $A$ is  $\alpha(A,q)/\binom{k}{t}$.
Then for practical purposes one wants ``high" $(q,t)$-completeness with few rows.

In these works, no theoretical results on partial coverage appear to have been stated; earlier contributions focus on  experimental investigations of
heuristic construction methods.
Our purpose is to initiate a mathematical investigation of arrays offering ``partial" coverage.
More precisely,   we address: 
\begin{itemize}
\item Can one obtain a significant improvement on the  upper bound (\ref{eq:can-upper}) if  the set $[v]^t$
is only required to be contained among the rows of \emph{at least}  $(1-\epsilon)\binom{k}{t}$ 
subarrays of $A$ of dimension $N\times t$? 
\item Can one obtain a significant improvement if, among the rows of \emph{every} $N\times t$ subarray of $A$,
only a (large) \emph{subset} of $[v]^t$ is required to be contained?
\item Can one obtain a significant improvement  if  the set $[v]^t$
is only required to be contained among the rows of \emph{at least}  $(1-\epsilon)\binom{k}{t}$ 
subarrays of $A$ of dimension $N\times t$, {\bf and} among the rows of each of the $\epsilon\binom{k}{t}$ subarrays that remain,  a (large) \emph{subset} of $[v]^t$ is required to be contained?
\end{itemize}
We answer these questions both theoretically and algorithmically in the following sections.

\section{Partial Covering Arrays}

When $1 \le m \le v^t$, a \emph{partial $m$-covering array},   $\PCA(N;t,k,v,m)$, is an $N \times k$ array $A$ with each entry  from $[v]$ so that for each $t$-set of columns $C \in {[k] \choose t}$, at least $m$ distinct tuples $x \in [v]^t$ appear as rows in $A_C$.
Hence a covering array $\CA(N;t,k,v)$ is precisely a partial $v^t$-covering array $\PCA(N;t,k,v,v^t)$.

\begin{theorem}\label{thm:pcan-bound1}
For integers $t,k,v$, and $m$ where $k \ge t \ge 2$, $v \ge 2$ and $1 \le m \le v^t$ there exists a $\PCA(N;t,k,v,m)$ with
\begin{equation}
N \le \frac{\ln \left\{{k \choose t}{v^t \choose m - 1}\right\}}{\ln \left(\frac{v^t}{m-1}\right)} .
\end{equation}.
\end{theorem}
\begin{proof}
Let $r = v^t - m + 1$, and $A$ be a random $N \times k$ array where each entry is chosen independently from $[v]$ with uniform probability. 
For $C \in {[k] \choose t}$, let $B_C$ denote the event that at least $r$ tuples from $[v]^t$ are missing in $A_C$. 
The probability that a particular $r$-set of tuples from $[v]^t$ is missing in $A_C$ is $\left(1 - \frac{r}{v^t}\right)^N$.
Applying the union bound to all $r$-sets of tuples from $[v]^t$, we obtain $\Pr[B_C] \le {v^t \choose r}\left(1 - \frac{r}{v^t}\right)^N$.
By linearity of expectation, the expected number of $t$-sets $C$ for which $A_C$ misses at least $r$ tuples from $[v]^t$ is at most ${k \choose t} {v^t \choose r}\left(1 - \frac{r}{v^t}\right)^N$.
When $A$ has at least $\frac{\ln \left\{{k \choose t}{v^t \choose m - 1}\right\}}{\ln \left(\frac{v^t}{m-1}\right)}$ rows this expected number is less than 1.
Therefore, an array $A$ exists with the required number of rows such that for all $C \in {[k] \choose t}$, $A_C$ misses at most $r-1$ tuples from $[v]^t$, i.e. $A_C$ covers at least $m$ tuples from $[v]^t$. 
\qed
\end{proof}

Theorem \ref{thm:pcan-bound1} can be improved upon using the Lov\'asz local lemma.
\begin{lemma}\label{lem:lllsym}
(Lov\'asz local lemma; symmetric case) (see {\rm \cite{alon08}}) 
Let $A_{1},A_{2},\ldots,A_{n}$ events in an arbitrary probability space. 
Suppose that each event $A_{i}$ is mutually independent of a set of all other events $A_{j}$ except for at most $d$, and that $\Pr[A_{i}]\le p$ for all $1\le i\le n$.
If $ep(d+1)\le1$, then $\Pr[\cap_{i=1}^{n}\bar{A_{i}}]>0$.
\end{lemma}
Lemma \ref{lem:lllsym} provides an upper bound on the probability of a ``bad'' event in terms of the dependence structure among such bad events, so that there is a guaranteed outcome in which all ``bad'' events are avoided. 
This lemma is most useful when there is limited dependence among the ``bad'' events, as  in the following:

\begin{theorem}\label{thm:pcan-bound2}
For integers $t,k,v$ and $m$ where $v,t \ge 2$, $k \ge 2t$ and $1 \le m \le v^t$ there exists a $\PCA(N;t,k,v,m)$ with
\begin{equation}\label{eq:pcan-bound}
N \le \frac{1 + \ln \left\{t{k \choose t - 1}{v^t \choose m - 1}\right\}}{\ln \left(\frac{v^t}{m-1}\right)} .
\end{equation}
\end{theorem}
\begin{proof}
When $k \ge 2t$, each event $B_C$ with $C \in {[k] \choose t}$ (that is, at least $v^t - m + 1$ tuples are missing in $A_C$)  is independent of all but at most ${t \choose 1}{k-1 \choose t-1}<t{k \choose t-1}$ events in $\{ B_{C'} : C' \in {[k] \choose t}\setminus \{C\}\}$. 
Applying Lemma \ref{lem:lllsym},  $\Pr[\wedge_{C \in {[k] \choose t}} \overline{B_C}]>0$ when 
\begin{equation}\label{eq:lll}
\mathrm{e}{v^t \choose r}\left(1 - \frac{r}{v^t}\right)^N t{k \choose t-1} \le 1.
\end{equation} 
Solve (\ref{eq:lll}) to obtain the required upper bound on $N$.
\qed
\end{proof}

When $m=v^t$, apply the Taylor series expansion to obtain $\ln \left(\frac{v^t}{m-1}\right) \ge \frac{1}{v^t}$, and thereby  recover the upper bound  (\ref{eq:can-upper}).
Theorem \ref{thm:pcan-bound2} implies:

\begin{corollary}
Given $q\in [0,1]$ and  integers $2 \le t \le k$, $v \ge 2$, there exists an $N\times k$ array on 
$[v]$ with  $(q, t)$-completeness
equal to 1 (i.e., \emph{maximal}), whose  number $N$ of rows satisfies 
$$N\leq \frac{1 + \ln \left\{t{k \choose t - 1}{v^t \choose qv^t - 1}\right\}}{\ln \left(\frac{v^t}{qv^t-1}\right)}.$$
\end{corollary}

Rewriting (\ref{eq:pcan-bound}), setting $r = v^t - m + 1$, and using the Taylor series expansion of $\ln \left(1 - \frac{r}{v^t}\right)$, we get
\begin{equation}\label{eq:pcan-bound-asymp}
N \le \frac{1 + \ln \left\{t{k \choose t - 1}{v^t \choose r}\right\}}{\ln \left(\frac{v^t}{v^t - r}\right)} \le \frac{v^t(t-1)\ln k}{r}\left\{1 - \frac{\ln r}{\ln k} + o(1)\right\}.
\end{equation}
Hence when $r = v(t-1)$ (or equivalently, $m = v^t - v(t-1) + 1$), there is a partial $m$-covering array with $\Theta(v^{t-1} \ln k)$ rows.
This matches the lower bound (\ref{eq:can-lower})
asymptotically for covering arrays by missing, in each $t$-set of columns, \emph{no more}
than $v(t-1)-1$ of the  $v^t$ possible rows.

The dependence of the bound (\ref{eq:pcan-bound}) on the number of $v$-ary $t$-vectors that must appear in the $t$-tuples of columns is particularly of interest when test suites are run sequentially until a fault is revealed, as in \cite{Bryce}.  Indeed the arguments here may have useful consequences for the rate of fault detection.

Lemma \ref{lem:lllsym} and hence Theorem \ref{thm:pcan-bound2} have proofs that are non-constructive in nature.
Nevertheless, Moser and Tardos \cite{moser10} provide a randomized algorithm with the same guarantee. 
Patterned on their method, Algorithm \ref{algo:m-t}  constructs a partial $m$-covering array with exactly the same number of rows as (\ref{eq:pcan-bound}) in expected polynomial time. 
Indeed, for fixed $t$, the expected number of times the resampling step (line \ref{line:re-sample}) is repeated is linear in $k$ (see \cite{moser10} for more details).

\begin{algorithm}[t]
\SetKw{Break}{break}
\KwIn{Integers $N,t,k,v$ and $m$ where $v,t \ge 2$, $k \ge 2t$ and $1 \le m \le v^t$}
\KwOut{$A$ : a $\PCA(N;t,k,v,m)$}
Let $N := \frac{1 + \ln \left\{t{k \choose t - 1}{v^t \choose m - 1}\right\}}{\ln \left(\frac{v^t}{m-1}\right)}$\; \label{line:lll-bound}
Construct an $N \times k$ array $A$ where each entry is chosen independently and uniformly at random from $[v]$\; 
\Repeat {covered $=$ true}{ \label{line:mt-loop}
    Set \emph{covered}$:=$ true\;
    \For {each column $t$-set $C \in {[k] \choose t}$}{ \label{line:online-check}
        \If {$A_C$ does not cover at least $m$ distinct $t$-tuples $x\in [v]^t$} {
            Set \emph{covered}$:=$ false\;
            Set \emph{missing-column-set} $:= C$\;
            \Break\;
        }
    }
    \If {covered $=$ false}{
        Choose all the entries in the $t$ columns of \emph{missing-column-set} independently and uniformly at random from $[v]$\; \label{line:re-sample}
    }
}
Output $A$\;

\caption{Moser-Tardos type algorithm for partial $m$-covering arrays.}
\label{algo:m-t}
\end{algorithm}

\section{Almost Partial Covering Arrays}

For $0 < \epsilon < 1$, an \emph{$\epsilon$-almost partial $m$-covering array},   $\APCA(N;t,k,v,m,\epsilon)$, is an $N \times k$ array $A$ with each entry from $[v]$ so that for at least $(1-\epsilon){k \choose t}$ column $t$-sets $C \in {[k] \choose t}$, $A_C$ covers at least $m$ distinct tuples $x \in [v]^t$.
Again,  a covering array $\CA(N;t,k,v)$ is precisely an  $\APCA(N;t,k,v,v^t, \epsilon)$ when $\epsilon < 1/ \binom{k}{t}$.
Our first result on $\epsilon$-{\em almost} partial $m$-covering arrays is the following.

\begin{theorem}\label{thm:apcan-bound}
For integers $t,k,v,m$ and real $\epsilon$ where $k \ge t \ge 2$, $v \ge 2$, $1 \le m \le v^t$ and $0 \le \epsilon \le 1$, there exists an $\APCA(N;t,k,v,m,\epsilon)$ with
\begin{equation}
N \le \frac{\ln \left\{{v^t \choose m - 1}/\epsilon\right\}}{\ln \left(\frac{v^t}{m-1}\right)}.
\end{equation}
\end{theorem}
\begin{proof}
Parallelling the proof of Theorem \ref{thm:pcan-bound1} we compute an upper bound on the expected number of $t$-sets $C\in {[k] \choose t}$ for which $A_C$ misses at least $r$ tuples $x \in [v]^t$. 
When this expected number is at most $\epsilon{k \choose t}$, an array $A$ is guaranteed to exist with at least $(1-\epsilon){k \choose t}$ $t$-sets of columns $C \in {[k] \choose t}$ such that $A_C$ misses at most $r-1$ distinct tuples $x \in [v]^t$. 
Thus $A$ is an $\APCA(N;t,k,v,m,\epsilon)$. 
To establish the theorem, solve the following  for $N$:
\begin{equation*}
{k \choose t} {v^t \choose r}\left(1 - \frac{r}{v^t}\right)^N \le \epsilon{k \choose t}.
\end{equation*}
\qed
\end{proof}

When $\epsilon < 1 / {k \choose t}$ we recover the bound from Theorem \ref{thm:pcan-bound1} for partial $m$-covering arrays.
In terms of $(q,t)$-completeness, Theorem \ref{thm:apcan-bound} yields the following.
\begin{corollary}
For $q\in [0,1]$ and  integers $2 \le t \le k$, $v \ge 2$,  there exists an $N\times k$ array on 
$[v]$ with  $(q, t)$-completeness
equal to $1-\epsilon$, with
$$N \leq \frac{\ln \left\{{v^t \choose m - 1}/\epsilon\right\}}{\ln \left(\frac{v^t}{m-1}\right)}.$$
\end{corollary}

When $m = v^t$, an $\epsilon$-almost covering array exists with $N \le v^t \ln \left(\frac{v^t}{\epsilon}\right)$ rows. 
Improvements result by focussing on covering arrays in which the symbols are acted on by a finite group.  
In this setting, one chooses orbit representatives of rows that collectively cover orbit representatives of $t$-way interactions under the group action; see \cite{ColCECA}, for example.
Such group actions have been used in direct and computational methods for covering arrays \cite{cck,MeagherS}, and  in randomized and derandomized methods \cite{ColCECA,SarkarColbourn2,sarkar16}.

We employ the sharply transitive action of the cyclic group of order $v$, adapting the earlier arguments using methods from \cite{sarkar16}: 

\begin{theorem}\label{thm:cyclic}
For integers $t,k,v$ and real $\epsilon$ where $k \ge t \ge 2$, $v \ge 2$ and $0 \le \epsilon \le 1$ there exists an $\APCA(N;t,k,v,v^t,\epsilon)$ with
\begin{equation}
N \le v^t \ln \left(\frac{v^{t-1}}{\epsilon}\right).
\end{equation}
\end{theorem}
\begin{proof}
The action of the cyclic group of order $v$ partitions $[v]^t$  into $v^{t-1}$ orbits, each of length $v$. 
Let $n = \lfloor \frac{N}{v} \rfloor$ and let $A$ be an $n \times k$ random array where each entry is chosen independently from the set $[v]$ with uniform probability. 
For $C \in {[k] \choose t}$,  $A_C$ \emph{covers the orbit} $X$ if at least one tuple $x\in X$ is present in $A_C$. 
The probability that the orbit $X$ is not covered in $A$ is $\left(1 - \frac{v}{v^t}\right)^n = \left(1 - \frac{1}{v^{t-1}}\right)^n$. 
Let $D_C$ denote the event  that $A_C$ does not cover at least one orbit. 
Applying the union bound,  $\Pr[D_C] \le v^{t-1}\left(1 - \frac{1}{v^{t-1}}\right)^n$.
By linearity of expectation, the expected number of column $t$-sets $C$ for which $D_C$ occurs is at most ${k \choose t}v^{t-1}\left(1 - \frac{1}{v^{t-1}}\right)^n$.
As earlier,  set this expected value to be at most $\epsilon{k \choose t}$ and solve for $n$. 
An array exists that covers all  orbits in at least $(1-\epsilon){k \choose t}$ column $t$-sets. 
Develop this array over the cyclic group to obtain the desired array.
\qed
\end{proof}

As in \cite{sarkar16}, further improvements result  by considering a group, like the Frobenius group, that acts sharply 2-transitively on $[v]$. 
When $v$ is a prime power, the \emph{Frobenius group} is the group of permutations of $\mathbb{F}_v$ of the form $\{x \mapsto ax+b\,:\,a,b\in \mathbb{F}_v,\,a\neq0\}$.

\begin{theorem}\label{thm:frob}
For integers $t,k,v$ and real $\epsilon$ where $k \ge t \ge 2$, $v \ge 2$, $v$ is a prime power and $0 \le \epsilon \le 1$ there exists an $\APCA(N;t,k,v,v^t,\epsilon)$ with
\begin{equation}
N \le v^t \ln \left(\frac{2v^{t-2}}{\epsilon}\right) + v.
\end{equation}
\end{theorem}
\begin{proof}
The action of the Frobenius group partitions $[v]^t$  into $\frac{v^{t-1}-1}{v-1}$ orbits of length $v(v-1)$ (full orbits) each and $1$ orbit of length $v$ (a short orbit).
The short orbit consists of tuples of the form $(x_1,\ldots,x_t)\in [v]^t$ where $x_1=\ldots=x_t$.
Let $n = \lfloor \frac{N-v}{v(v-1)}\rfloor$ and let $A$ be an $n \times k$ random array where each entry is chosen independently from the set $[v]$ with uniform probability. 
Our strategy is to construct $A$ so that it covers all full orbits for the required number of arrays $\{A_C :C \in {[k] \choose t}\}$. 
Develop $A$ over the Frobenius group and add $v$  rows of the form $(x_1, \ldots, x_k)\in[v]^t$ with $x_1= \ldots =x_k$ to obtain an $\APCA(N;t,k,v,v^t,\epsilon)$ with the desired value of $N$.
Following the lines of the proof of Theorem \ref{thm:cyclic},  $A$ covers all full orbits in at least $(1-\epsilon){k \choose t}$ column $t$-sets $C$ when
\[
{k \choose t}\frac{v^{t-1}-1}{v-1}\left(1 - \frac{v-1}{v^{t-1}}\right)^n \le \epsilon{k \choose t}.
\]
Because $\frac{v^{t-1}-1}{v-1} \le 2v^{t-2}$ for $v \ge 2$, we obtain the desired bound.
\qed
\end{proof}

Using  group action when $m=v^t$ affords useful improvements. 
Does this improvement extend to cases when $m < v^t$?
Unfortunately, the answer appears to be no.
Consider the case for $\PCA(N;t,k,v,m)$ when $m \le v^t$ using the action of the cyclic group of order $v$ on $[v]^t$. 
Let $A$ be a random $n \times k$ array over $[v]$. 
When $v^t-vs+1 \le m \le v^t-v(s-1)$ for $1 \le s \le v^{t-1}$, this implies that for all $C \in \binom{[k]}{t}$, $A_C$ misses at most $s-1$ orbits of $[v]^t$. 
Then we obtain that $n \le \left(1+\ln \left(t\binom{k}{t-1}\binom{v^{t-1}}{s}\right)\right)/\ln \left(\frac{v^{t-1}}{v^{t-1}-s}\right)$. 
Developing $A$ over the cyclic group we obtain a $\PCA(N;t,k,v,m)$ with 
\begin{equation}\label{eq:pcan-cyclic}
N \le v \frac{1+\ln \left\{\binom{k}{t-1}\binom{v^{t-1}}{s}\right\}}{\ln \left(\frac{v^{t-1}}{v^{t-1}-s}\right)}
\end{equation}

\begin{figure}
    \centering
    \begin{subfloat}[$t=6,\,k=20,\,v=4$]{
        \includegraphics[bb=100bp 238bp 520bp 553bp,clip,scale=0.42]{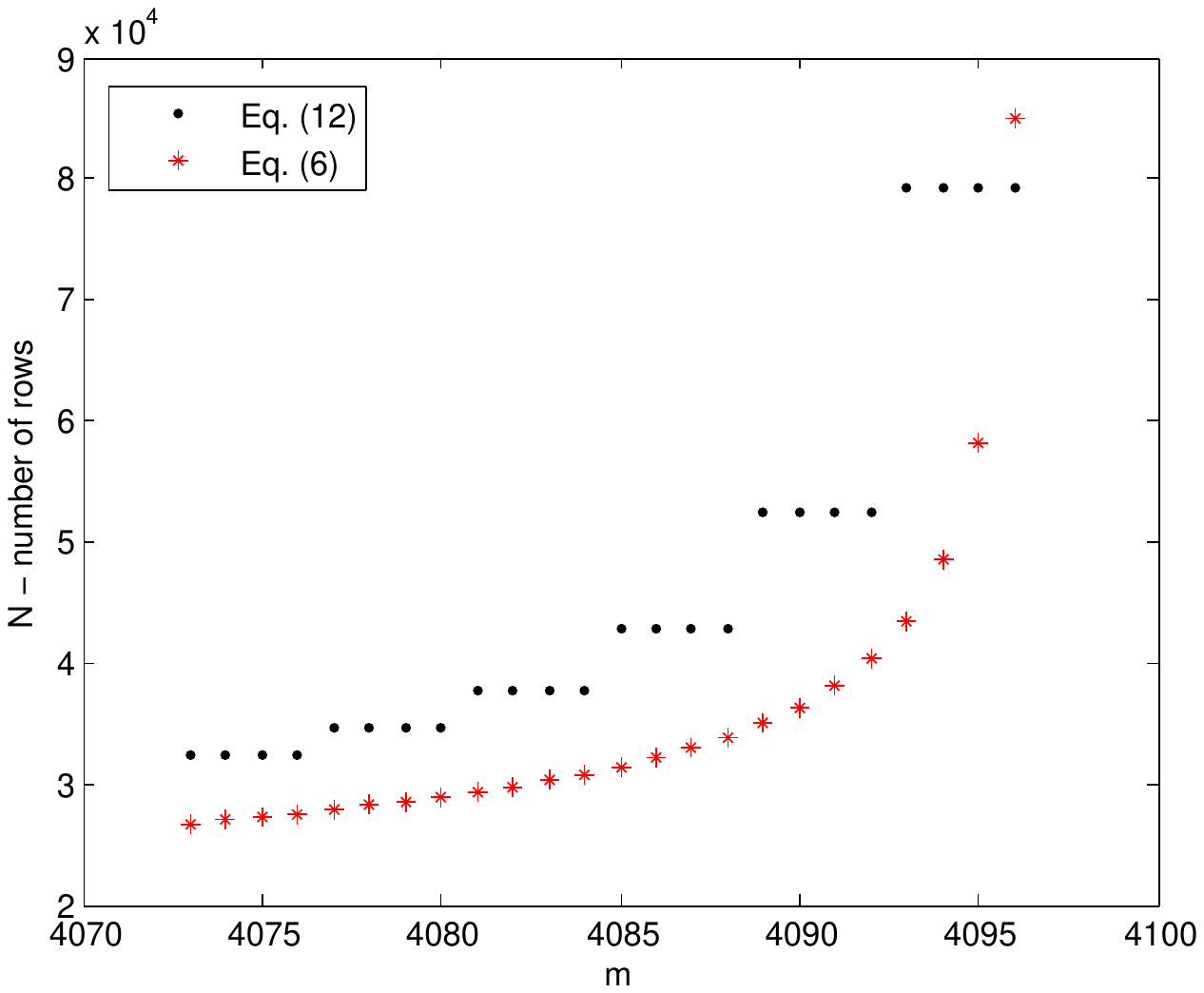}\label{fig:comp-ga-a}
        }
    \end{subfloat}
    \hspace{-0.35in} 
    \begin{subfloat}[$t=6,\,v=4,\,m=v^t-v$]{
        \includegraphics[bb=100bp 238bp 520bp 553bp,clip,scale=0.42]{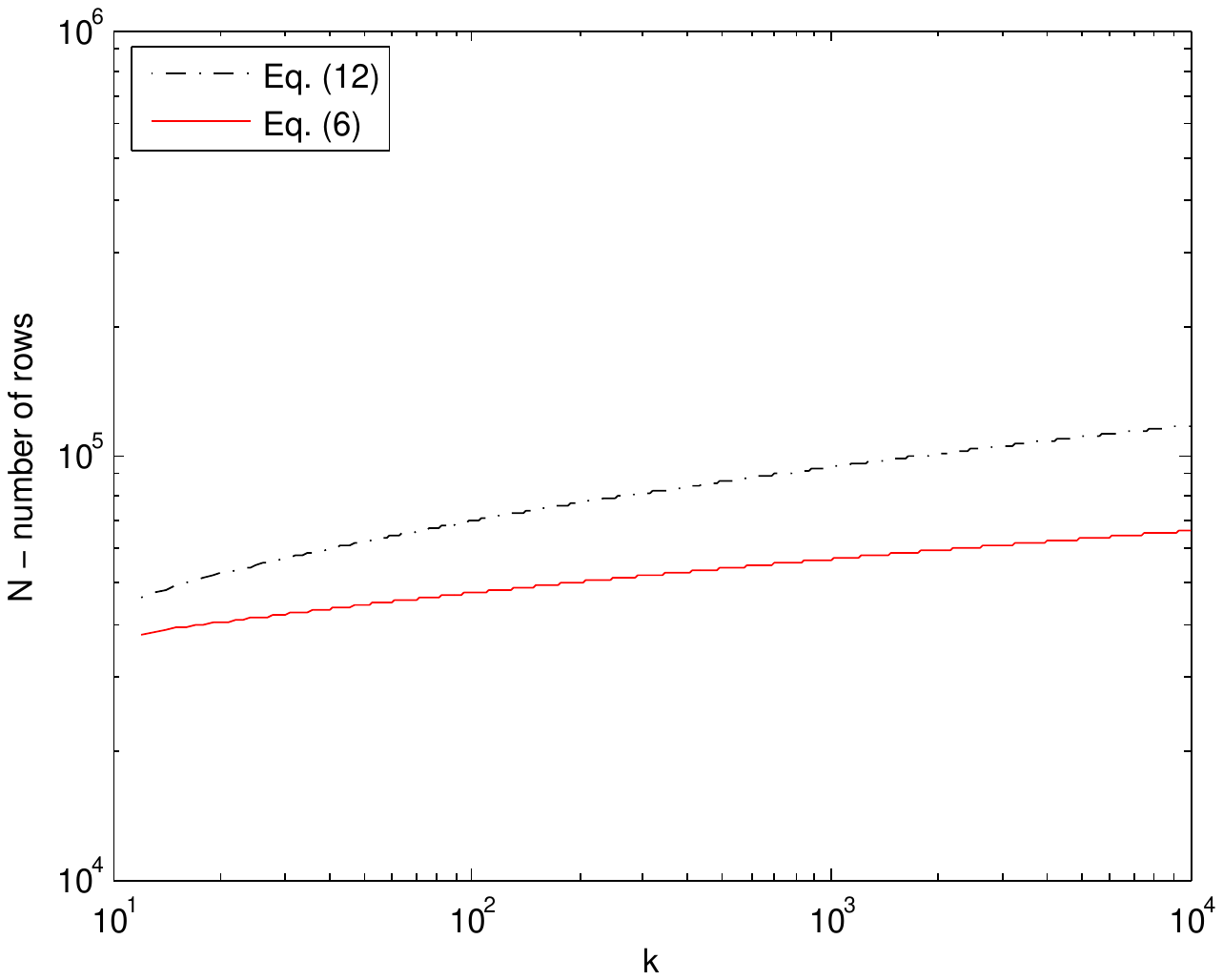}\label{fig:comp-ga-b}
        }
    \end{subfloat}
    \caption{Comparison of  (\ref{eq:pcan-cyclic}) and (\ref{eq:pcan-bound}). 
    Figure (a) compares the sizes of the partial $m$-covering arrays  when $v^t-6v+1 \le m \le v^t$. 
    Except for $m=v^t=4096$ the bound from  (\ref{eq:pcan-bound}) outperforms the bound obtained by assuming group action. 
    Figure (b) shows that for $m=v^t-v=4092$, (\ref{eq:pcan-bound})  outperforms  (\ref{eq:pcan-cyclic})  for all values of $k$.}\label{fig:comp-ga}
\end{figure}

Figure \ref{fig:comp-ga} compares  (\ref{eq:pcan-cyclic}) and  (\ref{eq:pcan-bound}). 
In Figure \ref{fig:comp-ga-a} we plot the size of the partial $m$-covering array as obtained by  (\ref{eq:pcan-cyclic}) and (\ref{eq:pcan-bound}) for $v^t-6v+1 \le m \le v^t$ and $t=6,\,k=20,\,v=4$. 
Except when $m=v^t=4096$,  the covering array case,  (\ref{eq:pcan-bound}) outperforms  (\ref{eq:pcan-cyclic}). 
Similarly, Figure \ref{fig:comp-ga-b} shows that for $m=v^t-v=4092$,  (\ref{eq:pcan-bound}) consistently outperforms  (\ref{eq:pcan-cyclic}) for all values of $k$ when $t=6,\,v=4$. 
We observe similar behavior for different values of $t$ and $v$. 

Next we consider even stricter coverage restrictions, combining Theorems \ref{thm:pcan-bound2} and \ref{thm:cyclic}.

\begin{theorem}\label{thm:concat}
For integers $t,k,v,m$ and real $\epsilon$ where $k \ge t \ge 2$, $v \ge 2$, $0 \le \epsilon \le 1$ and $m \le v^t + 1 - \frac{\ln k}{\ln (v/\epsilon^{1/(t-1)})}$ there exists an $N\times k$ array $A$ with entries from $[v]$ such that
\begin{enumerate}
\item for each $C \in {[k] \choose t}$, $A_C$ covers at least $m$ tuples $x\in[v]^t$,
\item for at least $(1 - \epsilon){k \choose t}$ column $t$-sets $C$, $A_C$ covers all tuples $x \in [v]^t$,
\item $N = O(v^t \ln \left(\frac{v^{t-1}}{\epsilon}\right))$.
\end{enumerate}
\end{theorem}
\begin{proof}
We  vertically juxtapose a partial $m$-covering array and an $\epsilon$-almost $v^t$-covering array.
For $r = \frac{\ln k}{\ln (v/\epsilon^{1/(t-1)})}$ and $m = v^t - r + 1$, (\ref{eq:pcan-bound-asymp}) guarantees the existence of a partial $m$-covering array with $v^t \ln \left(\frac{v^{t-1}}{\epsilon}\right)\{1+\mbox{o}(1)\}$ rows.
Theorem \ref{thm:cyclic}  guarantees the existence of an $\epsilon$-almost $v^t$-covering array with at most $v^t \ln \left(\frac{v^{t-1}}{\epsilon}\right)$ rows. 
\qed
\end{proof}

\begin{corollary}\label{cor:concat}
There exists an $N \times k$ array $A$ such that:
\begin{enumerate}
\item for any $t$-set of columns $C \in {[k] \choose t}$, $A_C$  covers at least $m \le v^t + 1 - v(t-1)$ distinct $t$-tuples $x\in [v]^t$,
\item for at least $\left(1-\frac{v^{t-1}}{k^{1/v}}\right){k \choose t}$ column $t$-sets $C$, $A_C$ covers all the distinct $t$-tuples $x\in [v]^t$.
\item $N = O(v^{t-1}\ln k)$.
\end{enumerate}
\end{corollary}
\begin{proof}
Apply Theorem \ref{thm:concat} with $m = v^t + 1 - \frac{\ln k}{\ln (v/\epsilon^{1/(t-1)})}$.
There are at most $\frac{\ln k}{\ln (v/\epsilon^{1/(t-1)})} -1$  missing $t$-tuples $x \in [v]^t$  in the $A_C$ for each of the at most $\epsilon{k\choose t}$  column $t$-sets $C$ that do not satisfy the second condition of Theorem \ref{thm:concat}.
To bound from above the number of missing tuples to  a certain small function $f(t)$ of  $t$,
it is sufficient that $\epsilon \le v^{t-1}\left(\frac{1}{k}\right)^\frac{t-1}{f(t)+1}$. 
Then the number of missing $t$-tuples $x \in [v]^t$ in $A_C$ is bounded from above by $f(t)$ whenever
$\epsilon$ is not larger than 
\begin{equation}\label{up}
v^{t-1}\left(\frac{1}{k}\right)^\frac{t-1}{f(t)+1}
\end{equation}

On the other hand, in order for the number $N=O\left(v^{t-1}\ln \left(\frac {v^{t-1}}{\epsilon}\right)\right)$ of rows of $A$ to be asymptotically equal to the lower bound  (\ref{eq:can-lower}), it suffices that $\epsilon$ is not smaller than
\begin{equation}\label{low}
 {v^{t-1}\over k^{\frac {1}{v}}}.
 \end{equation} 
When $f(t)=v(t-1)-1$,  (\ref{up}) and (\ref{low}) agree asymptotically, completing the proof.
\qed
\end{proof}

Once again we obtain a size that is $O(v^{t-1}\!\log k)$,  a goal that has not been reached for covering arrays.  
This is evidence that even a small relaxation of covering arrays provides arrays of the best sizes one can hope for.

Next we consider the efficient construction of the arrays whose existence is ensured by Theorem \ref{thm:concat}.
Algorithm \ref{algo:apca} is a randomized method to construct an $\APCA(N;t,k,v,m,\epsilon)$ of a size $N$ that is very close to the bound of Theorem \ref{thm:apcan-bound}.
By Markov's inequality the condition in line \ref{line:cond} of Algorithm \ref{algo:apca} is met with probability at most $1/2$. 
Therefore, the expected number of times the loop in line \ref{line:loop} repeats is at most $2$. 

To prove Theorem \ref{thm:apcan-bound}, $t$-wise independence among the variables is sufficient. 
Hence, Algorithm \ref{algo:apca} can be derandomized using $t$-wise independent random variables. 
We can also derandomize the algorithm using the method of conditional expectation. 
In this method we construct $A$ by considering the $k$ columns one by one and fixing all  $N$ entries of a column.
Given a set of already fixed columns, to fix the entries of the next column we consider all possible $v^N$ choices, and choose one that provides the maximum conditional expectation of the number of column $t$-sets $C \in \binom{[k]}{t}$ such that $A_C$ covers at least $m$ tuples $x\in[v]^t$. 
Because $v^N=O(\mathsf{poly}(1/\epsilon))$, this derandomized algorithm constructs the desired array in polynomial time.
Similar randomized and derandomized strategies can be applied to construct the array guaranteed by Theorem \ref{thm:cyclic}. 
Together with Algorithm \ref{algo:m-t} this implies that the array  in Theorem \ref{thm:concat} is also efficiently constructible.

\begin{algorithm}[t]
\SetKw{Break}{break}
\KwIn{Integers $N,t,k,v$ and $m$ where $v,t \ge 2$, $k \ge 2t$ and $1 \le m \le v^t$, and real $0<\epsilon<1$}
\KwOut{$A$ : an $\APCA(N;t,k,v,m,\epsilon)$}
Let $N :=\frac{\ln \left\{2{v^t \choose m - 1}/\epsilon\right\}}{\ln \left(\frac{v^t}{m-1}\right)}$\;  
\Repeat {isAPCA $=$ true}{ \label{line:loop}
	Construct an $N \times k$ array $A$ where each entry is chosen independently and uniformly at random from $[v]$\;
    Set \emph{isAPCA}$:=$ true\;
    Set \emph{defectiveCount}$:=$ 0\;
    \For {each column $t$-set $C \in {[k] \choose t}$} {
        \If {$A_C$ does not cover at least $m$ distinct $t$-tuples $x\in [v]^t$} {
            Set \emph{defectiveCount}$:=$ \emph{defectiveCount} + $1$\;
            \If {\emph{defectiveCount} $>$ $\lfloor\epsilon\binom{k}{t}\rfloor$}{ \label{line:cond}
				Set \emph{isAPCA}$:=$ false\;
            	\Break\;
            }
        }
    }
}
Output $A$\;

\caption{Randomized algorithm for $\epsilon$-almost partial $m$-covering arrays.}
\label{algo:apca}
\end{algorithm}

\section{Final Remarks}
We have  shown that by relaxing the  coverage requirement of a covering array somewhat, powerful upper bounds on the sizes of the  arrays can be established.
Indeed the upper bounds are  substantially smaller than the best known bounds for a covering array;  they are of the same order as the \emph{lower} bound for $\CAN(t,k,v)$.
As importantly, the techniques  not only provide asymptotic bounds but also  randomized polynomial time construction algorithms for such arrays. 

Our approach seems flexible 
enough to handle  variations of these problems. For instance, some applications require arrays that satisfy, for different
subsets of columns, different coverage or separation requirements \cite{Co2004}.
In \cite{GY2006} 
several interesting examples of  combinatorial problems are presented that can 
be unified and expressed in the framework  of $S$-constrained matrices.
Given a set of vectors $S$ each of length $t$, an $N\times k$ 
matrix $M$ is \emph{$S$-constrained} if for every $t$-set $C\in \binom{[k]}{t}$, $M_C$ contains as a row   each of the vectors in $S$.
The parameter to optimize is, as usual, the number of rows of $M$.
One potential direction is to ask for arrays that, in every $t$-tuple of columns, cover at least $m$ of the vectors in $S$, or that all vectors in $S$ are covered by all but a small number of $t$-tuples of columns.  Exploiting the structure of the members of $S$ appears to require an extension of the results developed here.

\section*{Acknowledgements}

Research of KS and CJC was supported in part by the National Science Foundation under Grant No. 1421058.

\bibliographystyle{plain}
\bibliography{CoveringArray}
\end{document}